\documentclass[12pt]{amsart}
\usepackage{amsmath,amssymb,amsbsy,amsfonts,latexsym,amsopn,amstext,cite,
                                               amsxtra,euscript,amscd,bm}
\usepackage{url}

\usepackage{mathrsfs}

\usepackage{color}
\usepackage[colorlinks,linkcolor=blue,anchorcolor=blue,citecolor=blue,backref=page]{hyperref}
\usepackage{color}
\usepackage{graphics,epsfig}
\usepackage{graphicx}
\usepackage{float}
\usepackage{epstopdf}
\hypersetup{breaklinks=true}

\usepackage[np]{numprint}
\npdecimalsign{\ensuremath{.}}

\usepackage{bibentry}

\usepackage[english]{babel}
\usepackage{mathtools}
\usepackage{todonotes}
\usepackage{url}
\usepackage[colorlinks,linkcolor=blue,anchorcolor=blue,citecolor=blue,backref=page]{hyperref}

\usepackage[norefs,nocites]{refcheck}

\begin{document}

\newcommand{\bs}{\boldsymbol}
\def \a{\alpha} \def \b{\beta} \def \d{\delta} \def \e{\varepsilon} \def \g{\gamma} \def \k{\kappa} \def \l{\lambda} \def \s{\sigma} \def \t{\theta} \def \z{\zeta}

\newcommand{\mb}{\mathbb}

\newtheorem{theorem}{Theorem}
\newtheorem{lemma}[theorem]{Lemma}
\newtheorem{claim}[theorem]{Claim}
\newtheorem{cor}[theorem]{Corollary}
\newtheorem{conj}[theorem]{Conjecture}
\newtheorem{prop}[theorem]{Proposition}
\newtheorem{definition}[theorem]{Definition}
\newtheorem{question}[theorem]{Question}
\newtheorem{example}[theorem]{Example}
\newcommand{\hh}{{{\mathrm h}}}
\newtheorem{remark}[theorem]{Remark}

\numberwithin{equation}{section}
\numberwithin{theorem}{section}
\numberwithin{table}{section}
\numberwithin{figure}{section}

\def\sssum{\mathop{\sum\!\sum\!\sum}}
\def\ssum{\mathop{\sum\ldots \sum}}
\def\iint{\mathop{\int\ldots \int}}

\newcommand{\diam}{\operatorname{diam}}

\newcommand{\supp}{\operatorname{supp}}

\newcommand{\ccr}[1]{{\color{magenta} #1}}
\newcommand{\ccm}[1]{{\color{magenta} #1}}
\newcommand{\ccc}[1]{{\color{cyan} #1}}
\newcommand{\cco}[1]{{\color{orange} #1}}

\def\squareforqed{\hbox{\rlap{$\sqcap$}$\sqcup$}}
\def\qed{\ifmmode\squareforqed\else{\unskip\nobreak\hfil
\penalty50\hskip1em \nobreak\hfil\squareforqed
\parfillskip=0pt\finalhyphendemerits=0\endgraf}\fi}

\newfont{\teneufm}{eufm10}
\newfont{\seveneufm}{eufm7}
\newfont{\fiveeufm}{eufm5}
%
%
\newfam\eufmfam
     \textfont\eufmfam=\teneufm
\scriptfont\eufmfam=\seveneufm
     \scriptscriptfont\eufmfam=\fiveeufm
%
%
\def\frak#1{{\fam\eufmfam\relax#1}}

\newcommand{\bflambda}{{\boldsymbol{\lambda}}}
\newcommand{\bfmu}{{\boldsymbol{\mu}}}
\newcommand{\bfxi}{{\boldsymbol{\eta}}}
\newcommand{\bfrho}{{\boldsymbol{\rho}}}

\def\eps{\varepsilon}

\def\fK{\mathfrak K}
\def\fT{\mathfrak{T}}
\def\fL{\mathfrak L}
\def\fR{\mathfrak R}
\def\fQ{\mathfrak Q}

\def\fA{{\mathfrak A}}
\def\fB{{\mathfrak B}}
\def\fC{{\mathfrak C}}
\def\fM{{\mathfrak M}}
\def\fS{{\mathfrak  S}}
\def\fU{{\mathfrak U}}

\def\sssum{\mathop{\sum\!\sum\!\sum}}
\def\ssum{\mathop{\sum\ldots \sum}}
\def\dsum{\mathop{\quad \sum \qquad \sum}}
\def\iint{\mathop{\int\ldots \int}}
 
\def\T {\mathsf {T}}
\def\Tor{\mathsf{T}_d}
\def\Tore{\widetilde{\mathrm{T}}_{d} }

\def\sM {\mathsf {M}}
\def\sL {\mathsf {L}}
\def\sK {\mathsf {K}}
\def\sP {\mathsf {P}}

\def\ss{\mathsf {s}}

\def \balpha{\bm{\alpha}}
\def \bbeta{\bm{\beta}}
\def \bgamma{\bm{\gamma}}
\def \bdelta{\bm{\delta}}
\def \bzeta{\bm{\zeta}}
\def \blambda{\bm{\lambda}}
\def \bchi{\bm{\chi}}
\def \bphi{\bm{\varphi}}
\def \bpsi{\bm{\psi}}
\def \bxi{\bm{\xi}}
\def \bnu{\bm{\nu}}
\def \bomega{\bm{\omega}}

\def \bell{\bm{\ell}}

\def\eqref#1{(\ref{#1})}

\def\vec#1{\mathbf{#1}}

\newcommand{\abs}[1]{\left| #1 \right|}

\def\Zq{\mathbb{Z}_q}
\def\Zqx{\mathbb{Z}_q^*}
\def\Zd{\mathbb{Z}_d}
\def\Zdx{\mathbb{Z}_d^*}
\def\Zf{\mathbb{Z}_f}
\def\Zfx{\mathbb{Z}_f^*}
\def\Zp{\mathbb{Z}_p}
\def\Zpx{\mathbb{Z}_p^*}
\def\cM{\mathcal M}
\def\cE{\mathcal E}
\def\cH{\mathcal H}

\def\le{\leqslant}
\def\leq{\leqslant}
\def\ge{\geqslant}
\def\leq{\leqslant}

\def\sfB{\mathsf {B}}
\def\sfC{\mathsf {C}}
\def\sfS{\mathsf {S}}
\def\sfI{\mathsf {I}}
\def\sfT{\mathsf {T}}
\def\L{\mathsf {L}}
\def\FF{\mathsf {F}}

\def\sE {\mathscr{E}}
\def\sS {\mathscr{S}}

\def\cA{{\mathcal A}}
\def\cB{{\mathcal B}}
\def\cC{{\mathcal C}}
\def\cD{{\mathcal D}}
\def\cE{{\mathcal E}}
\def\cF{{\mathcal F}}
\def\cG{{\mathcal G}}
\def\cH{{\mathcal H}}
\def\cI{{\mathcal I}}
\def\cJ{{\mathcal J}}
\def\cK{{\mathcal K}}
\def\cL{{\mathcal L}}
\def\cM{{\mathcal M}}
\def\cN{{\mathcal N}}
\def\cO{{\mathcal O}}
\def\cP{{\mathcal P}}
\def\cQ{{\mathcal Q}}
\def\cR{{\mathcal R}}
\def\cS{{\mathcal S}}
\def\cT{{\mathcal T}}
\def\cU{{\mathcal U}}
\def\cV{{\mathcal V}}
\def\cW{{\mathcal W}}
\def\cX{{\mathcal X}}
\def\cY{{\mathcal Y}}
\def\cZ{{\mathcal Z}}
\newcommand{\rmod}[1]{\: \mbox{mod} \: #1}

\def\cg{{\mathcal g}}

\def\vy{\mathbf y}
\def\vr{\mathbf r}
\def\vx{\mathbf x}
\def\va{\mathbf a}
\def\vb{\mathbf b}
\def\vc{\mathbf c}
\def\ve{\mathbf e}
\def\vf{\mathbf f}
\def\vg{\mathbf g}
\def\vh{\mathbf h}
\def\vk{\mathbf k}
\def\vm{\mathbf m}
\def\vz{\mathbf z}
\def\vu{\mathbf u}
\def\vv{\mathbf v}

\def\e{{\mathbf{\,e}}}
\def\ep{{\mathbf{\,e}}_p}
\def\eq{{\mathbf{\,e}}_q}

\def\Tr{{\mathrm{Tr}}}
\def\Nm{{\mathrm{Nm}}}

 \def\SS{{\mathbf{S}}}

\def\lcm{{\mathrm{lcm}}}

 \def\0{{\mathbf{0}}}

\def\({\left(}
\def\){\right)}
\def\l|{\left|}
\def\r|{\right|}
\def\fl#1{\left\lfloor#1\right\rfloor}
\def\rf#1{\left\lceil#1\right\rceil}
\def\sumstar#1{\mathop{\sum\vphantom|^{\!\!*}\,}_{#1}}

\def\mand{\qquad \mbox{and} \qquad}

\def\tblue#1{\begin{color}{blue}{{#1}}\end{color}}




\hyphenation{re-pub-lished}

\mathsurround=1pt

\def\bfdefault{b}

\def \F{{\mathbb F}}
\def \K{{\mathbb K}}
\def \N{{\mathbb N}}
\def \Z{{\mathbb Z}}
\def \P{{\mathbb P}}
\def \Q{{\mathbb Q}}
\def \R{{\mathbb R}}
\def \C{{\mathbb C}}
\def\Fp{\F_p}
\def \fp{\Fp^*}

 \def \xbar{\overline x}

\title{Additive energy of polynomial images}

 \author[B.\ Kerr]{Bryce Kerr}
\address{B.K.:  School of Science, University of New South Wales,
Canberra,  ACT 2610, Australia}
\email{bryce.kerr@unsw.edu.au}

 \author[A. Mohammadi] {Ali Mohammadi}
\address{A.M.: Department of Pure Mathematics, University of New South Wales,
Sydney, NSW 2052, Australia}
\email{ali.mohammadi.np@gmail.com}

 \author[I. E. Shparlinski] {Igor E. Shparlinski}
\address{IES: Department of Pure Mathematics, University of New South Wales,
Sydney, NSW 2052, Australia}
\email{igor.shparlinski@unsw.edu.au}

\begin{abstract}   Given a monic polynomial $f(X)\in \Z_m[X]$ over a residue ring $\Z_m$ modulo an 
integer $m\ge 2$ and a discrete interval  $\cI = \{1, \ldots, H\}$ of $H \le m$ consecutive integers, considered as elements of $\Z_m$, 
we obtain a new upper bound for the additive energy of the set $f(\cI)$, where $f(\cI)$ denotes the image set  $f(\cI)  = \{f(u):~u \in  \cI\}$. 
We give an application of our bounds to multiplicative character sums, improving some previous result of 
Shkredov and Shparlinski~(2018).
\end{abstract}


\maketitle

\tableofcontents

\section{Introduction} 

\subsection{Background}
Given a polynomial 
\begin{equation}
\label{eq:Poly f}
f(X)= a_d X^d +\ldots +a_1X + a_0\in \Z_m[X]
\end{equation} 
 of degree $d \ge 2$ over a residue ring $\Z_m$  and a discrete interval  
\begin{equation}
\label{eq:Int I}
\cI = \{1, \ldots, H\}
\end{equation}  of $H \le m$ consecutive integers
we consider the image set 
 $$
 f(\cI)  = \{f(u):~u \in  \cI\} \subseteq \Z_m
 $$ 
 of $\cI$, where all calculations are carried-out in $\Z_m$. 
 
 The combinatorial structure of the set $f(\cI)$ has been studied in a number 
 of work. For example, Castro and Chang~\cite[Theorem~1.1]{CaCh} have estimated the multiplicative energy 
 $$
 E_{f,m}^\times(\cI) = \# \{(v,w,x,y)\in  f(\cI)^4:~ vw = xy\}
 $$
 (for highly composite square-free moduli $m$), while Shkredov and 
 Shparlinski~\cite[Corollary~2.12]{ShkShp} obtained upper bounds on the additive energy 
 $$
 E_{f,m}^+(\cI) =  \# \{(v,w,x,y) \in  f(\cI)^4:~ v+w = x+y\}
 $$
 (for prime moduli $m$). 
 
 We also note that since we work in a ring rather than a field, the standard trivial bound 
$$ 
\# f(\cI) \ge d^{-1} H
$$
generally speaking cannot be asserted for highly composite moduli 
$m$ (see~\cite{Kon,KonSte}  
about the number of 
solutions to polynomial  congruences). 
In fact, it is more convenient to work with a slightly different quantity
\begin{align*}
T_{f,m}(\cI) = \# \{ (x,y,z,w)&\in \cI^4:\\
& ~f(x) + f(y)\equiv f(z)+f(w) \mod {m}\}
\end{align*}
for which we have 
$$
 E_{f,m}^+(\cI)  \le T_{f,m}(\cI) 
 $$ 
and is of the same order of magnitude as  $E_{f,m}^+(\cI)$ for a prime $m =p$. In fact
$$
E_{f,p}^+(\cI)  \ge  d^{-4} T_{f,p}(\cI).
$$
Furthermore, if $|\cI|\le m^{1/d,}$ it follows from work of Konyagin  and Steger~\cite{KonSte} that $E_{f,m}^+(\cI)$ and $T_{f,m}(\cI)$ are the same order of magnitude.

We note that at least for a prime $m$,  the upper bound on  the additive energy 
from Shkredov and Shparlinski~\cite[Corollary~2.12]{ShkShp} implies a lower bound on $\#\(f(\cI) + f(\cI)\)$.  However, here we use a different approach, 
which originates from~\cite{CCGHSZ},  and allows one to obtain a stronger bound. 
 
 There are also several previous works studying combinatorial properties of  $f(\cI)$. In particular, 
 intersections of $f(\cI)$ with other structural sets, such as
a discrete interval or a multiplicative subgroup of $\Z_m^*$ are considered in~\cite{Chang,CCGHSZ,CGOS,Go-PSh,IKSSS,KeMo,MeSh,Shp}. 
 Similar problems in finite fields have also  been studied in~\cite{Ost,RNS}.

\subsection{Main Results}


Our first  bound is  nontrivial as long as $H \le m^{1-\varepsilon}$
where $\varepsilon > 0$ can be arbitrary.

 \begin{theorem}
\label{thm:Energy1} Let $f \in \Z_m[X]$ be a polynomial of degree $d \ge 2$ 
as in~\eqref{eq:Poly f} with $\gcd(a_d,m)=1$ and let $\cI$ be an interval   as in~\eqref{eq:Int I}.  Then
$$
 T_{f,m}(\cI) \le H^{3+o(1)}  \min\left\{\(m/H\)^{-\alpha_d}, H^{-\beta_d} \right\} , 
$$
where 
\begin{equation}
\label{eq: alpha beta}
\alpha_d = \frac{2}{d^2+d -2} \mand \beta_d =  \frac{2}{d +2}.
\end{equation} 
\end{theorem}

For smaller intervals we obtain a sharper bound which generalises the estimate into the setting of the energy of $f(\cI)$ with an arbitrary set. Recall that for two sets $\cA,\cB$
$$E(\cA,\cB)=\# \{ a_1,a_2\in \cA, \   b_1,b_2\in \cB :~ a_1+b_1=a_2+b_2\}.$$

\begin{theorem}
\label{thm:Energy2}
Let $f \in \Z_m[X]$ be a polynomial of degree $d \ge 2$ 
as in~\eqref{eq:Poly f} with $\gcd(a_d,m)=1$ and let $\cI$ an interval   as in~\eqref{eq:Int I}.  
Then for any $\cZ \subseteq \Z_m$ of cardinality $\#\cZ = Z$ we have 
$$
 E(\cZ ,f(\cI)) \le \(\frac{H^{2}Z^2}{m^{2/d(d+1)}}+Z(H+Z)\)H^{o(1)}. 
$$
\end{theorem}

Considering $T_{f,m}(\cI)$ as a special case of $E(\cZ ,f(\cI))$ with $Z=f(\cI)$ in Theorem~\ref{thm:Energy2}, we obtain a sharper bound for longer intervals $\cI$.

\begin{theorem}
\label{thm:main3}
Let  $f \in \Z_m[X]$ be a polynomial of degree $d \ge 2$ 
as in~\eqref{eq:Poly f} with $\gcd(a_d,m)=1$ and let $\cI$ be an interval as in~\eqref{eq:Int I}. We have 
$$
T_{f,m}(\cI)\le \(\frac{H^{4}}{m^{4/d(d+1)}}+ H^{2} \)H^{o(1)}.
$$
\end{theorem}
In particular, if 
$$H\le m^{2/d(d+1)},$$
then  the bound of Theorem~\ref{thm:Energy2} becomes 
$$
 T_{f,m}(\cI)  \le H^{2+o(1)}
$$
which is optimal up to the factor $H^{o(1)}$.

%
%
%
%
%

\subsection{Application to character sums}
Let $\chi$ denote a non-principal multiplicative character modulo a prime $p$. As an application of Theorem~\ref{thm:main3}, we prove the following bounds on character sums over primes, which improve on results from~\cite[Theorem~1.7]{ShkShp}.

\begin{theorem}
\label{thm:Primes} Let $f\in \Z_p[X]$ be a  polynomial of degree $d\ge  2$. 
For any $Q=p^{\zeta + o(1)}$ and $R=p^{\xi + o(1)}$ with some fixed positive $\zeta$ and  $ \xi\le \min\{1/2, 2 - 2 \zeta\}$ satisfying
$$
\zeta+\xi >1/2, \qquad \xi>1/2 - 2/d(d+1), \qquad  \zeta+5\xi/2 > 1,
$$
there exists some $\delta>0$ such that 
$$
  \sum_{\substack{q\le Q\\q~\text{prime}}}  \left|  \sum_{\substack{r \le R\\ r~\text{prime}}}
 \chi(f(q) +r)\right|\,, \  \sum_{\substack{r \le R\\ r~\text{prime}}}  \left| \sum_{\substack{q\le Q\\ q~\text{prime}}} 
 \chi(f(q) +r )\right|  
 \le  QR   p^{-\delta}\,. 
$$
\end{theorem}

We note that~\cite[Theorem~1.7]{ShkShp} gives the same result under the condition $5\zeta/4+2\xi >1$ for $d=2$ and $(1 + 2^{-d+1})\zeta +2\xi>1$ for $d\geq 3$ in place of  the conditions  $\xi>1/2 - 2/d(d+1)$ 
and $\xi+\zeta>1/2$ in  Theorem~\ref{thm:Primes} (the conditions $ \xi\le \min\{1/2, 2 - 2 \zeta\}$ and $ \zeta+5\xi/2 > 1$ 
are  the same in all results). 
For example,  if  $\zeta =   1/4$, for for $d=2$ and $d=3$  
we now have a nontrivial result if $\xi > 3/10$ and $\xi > 1/3$, respectively
while~\cite[Theorem~1.7]{ShkShp}  is nontrivial only for $\xi \ge 11/32$ (for both values of $d$).

 \subsection{Notation} 
We recall that the notations $U = O(V)$, 
$U \ll V$ and $ V\gg U$  are equivalent to $|U|\leqslant c|V| $ for some positive constant $c$, 
which throughout the paper may depend on the degree $d$ of $f$, the integer parameter $s \ge 1$, 
and the real parameter $\varepsilon > 0$. 

We denote the cardinality of a finite set $\cS$ by $\#\cS$.

For a real number $z$, define $\e(z) = \exp(2 \pi i z)$.

\section{Preliminaries} 

 \subsection{A general form of the Vinogradov mean value theorem}
\label{sec:MVT}

The following is a consequence of results of Bourgain, Demeter and Guth~\cite[Theorem~4.1]{BDG} and of Wooley~\cite[Theorem~1.1]{Wool}.

\begin{lemma}
\label{lem:vmvt}
Let $d$ be an integer, $s$ a positive real number with $s\le d(d+1)/2$ and $(a_n)_{n\in \Z}$ a sequence of complex numbers. We have 
\begin{align*}
\int_{[0,1)^{d}} \left|\sum_{|n|\le H}a_n \e(\alpha_1 n+\ldots+\alpha_d n^{d}) \right|^{2s} &d\alpha_1\ldots d\alpha_{d} \\
& \ll H^{o(1)}\(\sum_{|n|\le H }|a_n|^2 \)^{s}.
\end{align*}
\end{lemma}

Let $\cX\subseteq [1,H]\cap \Z$ be an arbitrary set. For integers $d$ and $s$ we let $J_{k,s}(\cX)$ denote the number of solutions to the system of equations
$$
x_1^{j}+\ldots+x_{s}^{j}=x_{s+1}^{j}+\ldots+x_{2s}^{j}, \qquad 1\le j \le d,
$$
with variables satisfying
$$x_1,\ldots,x_{2s}\in \cX.$$

The following is an immediate corollary of Lemma~\ref{lem:vmvt}.

\begin{lemma}
\label{lem:MVT-GenSet}
For $s\le d(d+1)/2$, for and set $\cX\subseteq [1,H]\cap \Z$ of cardinality $X$
$$J_{d,s}(\cX)\ll  X^{s}H^{o(1)}.$$
\end{lemma}

\subsection{Background on  geometry of numbers}
We next recall some facts from the geometry of numbers. Given a lattice $\cL \subseteq \R^{n}$ and a symmetric convex body $\cD\subseteq \R^{n}$, define the $i$-th successive minimum of $\cL$ with respect to $\cD$ by 
$$\lambda_{i}=\inf\{ \lambda :~\cL \cap \lambda \cD \  \text{contains $i$ linearly independent points} \}.$$

The following basis is  known as a Mahler basis, see~\cite[Theorem 3.34, Corollary 3.35]{TaoVu}.
\begin{lemma}
\label{lem:mahler}
Given a  lattice  $\cL\subseteq \R^{n}$  
and  a  convex body $\cB\subseteq \R^{n}$, let $\lambda_{1},\ldots,\lambda_{n}$ denote the successive minima of $\cL$ with respect to $\cB$.  There exists a basis $\vec{w}_1,\ldots,\vec{w}_n$ of $\cL$ such that 
$$\vec{w}_j\in \frac{n\lambda_j}{2}\cdot \cB,$$
and each $\vec{b}  \in \cL\cap \cB$ can be expressed in the form 
$$\vec{b}=b_1\vec{w}_1+\ldots+ b_n \vec{w}_n, \qquad  b_j\ll \frac{1}{\lambda_j}, \ j =1, \ldots, n.$$
\end{lemma}

The following is Minkowski's second theorem, for a proof see~\cite[Theorem~3.30]{TaoVu}.

\begin{lemma}
\label{lem:mst}
Given a  lattice  $\cL\subseteq \R^{n}$  and  a symmetric convex body $\cD\subseteq \R^{n}$, let $\lambda_{1},\ldots,\lambda_{n}$ denote the successive minima of $\cL$ with respect to $\cD$. We have
$$\frac{\mu(\cD)}{\mu(\R^n/\cL)}\ll \frac{1}{\lambda_1\ldots\lambda_n}  \ll\frac{\mu(\cD)}{\mu(\R^n/\cL)},$$
where $\mu$ denotes the Lebesgue  measure.
\end{lemma}

 We may bound the number of lattice points $\#\(\cL \cap \cD\)$ in terms of the successive minima, see for example~\cite[Exercise~3.5.6]{TaoVu}.
 
\begin{lemma}
\label{lem:lattice}
Given a  lattice  $\cL\subseteq \R^{n}$  
and  a symmetric convex body $\cD\subseteq \R^{n}$, let $\lambda_{1},\ldots,\lambda_{n}$ denote the successive minima of $\cL$ with respect to $\cD$. We have
$$\#\(\cL \cap \cD\)\ll \prod_{j=1}^{n}\max\left\{1, \frac{1}{\lambda_j} \right\}.$$
\end{lemma}

Define the dual lattice $\cL^{*}$ by
\begin{equation}
\label{eq:Gammastar}
\cL^{*}=\{ y\in \R^{n}:~\langle y,z\rangle\in \Z \ \text{for all $z\in \cL$}\},
\end{equation}
and the dual body $\cD^{*}$ by
\begin{equation}
\label{eq:Dstar}
\cD^{*}=\{ y \in \R^{n} :~ \langle y,z \rangle \le 1 \ \text{for all $z\in \cD$} \},
\end{equation}
where $\langle   \cdot  ,  \cdot   \rangle$ denotes the Euclidean inner product. 

The successive minima of a lattice with respect to a convex body and the successive minima of the dual lattice with respect to the dual body are related through the following estimates, see for example~\cite{Ba}.

\begin{lemma}
\label{lem:transfer}
Given a  lattice  $\cL\subseteq \R^{n}$  
and  a symmetric convex body $\cD\subseteq \R^{n}$,  let $\lambda_{1},\ldots,\lambda_{n}$ denote the successive minima of $\cL$ with respect to $\cD$ and  let $\lambda_1^{*},\ldots,\lambda_n^{*}$ denote the successive minima of the dual lattice $\cL^{*}$ with respect to the dual body $\cD^{*}$.  For each $1\le j \le n$ we have
$$\lambda_j \lambda^*_{n-j+1}\ll 1.$$
\end{lemma} 


We also require the following generalisation of Siegel's Lemma due to Bombieri and Vaaler~\cite[Theorem~2]{BomVaa}.
\begin{lemma}\label{lem:BomVaa}
    Let $\sum_{i=1}^d v_{j, i}w_i = 0$, $j=1,\ldots , d_0$ be a system of $d_0$ linearly independent equations in $d>d_0$ unknowns, with rational integer coefficients $v_{j,i}$. Then, there are $d-d_0$ linearly independent integral solutions $\tilde w_j=(w_{j,1},\ldots ,w_{j,d})$, $1\leq j\leq d-d_0$, with
    \[
    \prod_{j=1}^{d-d_0}\max_i |w_{j,i}|\leq \left(D^{-1} \sqrt{|\det(M M^T)|}\right)^{1/(d-d_0)},
    \]
    where $M$ denotes the $d_0\times d$ matrix $M=(v_{j,i})$ and $D$ is the greatest common divisor of determinants of all $d_0\times d_0$ minors of $M$.
\end{lemma}

\subsection{Additive equations with polynomials}

\begin{lemma}
\label{lem:poly-eqn}
Let $f\in \Z[X]$ be a polynomial of degree $d\ge 2$. For any $w\neq 0$, the number of solutions to the equation
\begin{equation}
\label{eq:fnm}
 f(n)-f(m)=w, \quad 1\le n,m\le H,
\end{equation}
is bounded by $w^{o(1)}$.
\end{lemma}
\begin{proof}
By Taylor's formula, it follows that
$$f(x)-f(y)=(x-y)g(x,y),$$
where for any $t$, the polynomial $G_t(x)=g(x,x+t)$ is not constant.

For any solution $n,m$ to~\eqref{eq:fnm}, there exists $d_1,d_2$ satisfying 
$$d_1d_2=w,$$ 
and 
\begin{equation}
\label{eq:nm1}
n-m=d_1, \quad g(n,m)=d_2.
\end{equation}
Let $R$ count the number of solutions to~\eqref{eq:fnm}. The well known bound on the divisor function (see, for example,~\cite[Equation~(1.81)]{IwKow}) implies that
$$R\le w^{o(1)}\max_{d_1,d_2\in \Z}\#\{ 1\le n,m\le H :~ n-m=d_1, \ g(n,m)=d_2\}.$$
If $n,m$ satisfy~\eqref{eq:nm1}, then 
$$g(d_1+m,m)=d_2,$$
for which there are $O(1)$ solutions in variable $m$. For each such $m$, there exists at most one solution in remaining variable $n$. This establishes the desired result.
\end{proof}

We now immediately derive the following.

\begin{cor}
\label{cor:poly-eqn}
Let $f\in \Z[X]$ be a polynomial of degree $d\ge 2$ with coefficients bounded by $X$. For any $H\ge 1$, we have  
$$
\#\{ 1\le x,y,z,w\le H : ~f(x)+f(y)=f(z)+f(w)\}\le H^2 (HX)^{o(1)}.
$$
\end{cor}
\begin{proof}
The result follows from Lemma~\ref{lem:poly-eqn} after noting 
\begin{align*}
\#\{ 1\le x,y,z,w\le H:~f(x)+f(y)=f(z)+f(w)\}& \\
\ll H^2&+\sum_{w\neq 0}r(w)^2,
\end{align*} 
where 
$$r(w)=\#\{ 1\le x,y\le H:~f(x)-f(y)=w\}.$$
In particular, since if $r(w)\neq 0$ then $|w|\ll XH^{d}$ and hence by  Lemma~\ref{lem:poly-eqn}
$$
\sum_{w\neq 0}r(w)^2\le (XH)^{o(1)}\sum_{w\neq 0}r(w)\le H^2 (XH)^{o(1)}.
$$
\end{proof}

We next prove an analogue of  Lemma~\ref{lem:poly-eqn} for congruences. 

\begin{lemma}
\label{lem:short}
Let $m$ be an integer, $f$ a polynomial of degree $d\ge 2$ with leading coefficient $a_d$ coprime to $m$. There exists a constant $c_d$ depending only on $d$ such that if 
\begin{equation}
\label{eq:short-H}
H\le c_d m^{2/d(d+1)}
\end{equation}
 then  for any $\lambda \not \equiv 0 \mod{m}$ the number of solutions to 
\begin{equation}
\label{eq:cong11}
f(n)-f(m)\equiv \lambda \mod{m}, \quad 1\le m,n\le H
\end{equation}
is bounded by $H^{o(1)}$.
\end{lemma}
\begin{proof}
Let $W$ count the number of solutions to~\eqref{eq:cong11}.

Suppose $f$ is given by 
$$f(x)=\sum_{j=1}^{d}a_j x^{j}.$$
Since $(a_d,m)=1$, by modifying $\lambda$ if necessary we may assume that $a_d=1$. 

 Define the lattice 
$$\cL=\{ (n_1,\ldots ,n_d)\in \Z^{d}:~\exists  \ell \  \text{such that} \  a_j\ell \equiv n_j \mod m \ 1\le j \le d\},$$
and convex body 
$$\cB=\left\{ (x_1,\ldots ,x_d)\in \R^{d} :~|x_j|\le \frac{m}{100 d H^{j}}\right\}.$$
 By Lemma~\ref{lem:mst} and the condition~\eqref{eq:short-H} 
$$
\cL\cap \cB \neq \{0\}.
$$
Let $(b_1,\ldots ,b_d)\in \cL\cap \cB/\{0\}$ and choose $\ell\not \equiv 0 \mod{m}$ to satisfy 
\begin{equation}
\label{eq:123}
a_j\ell \equiv b_j \mod{m}, \quad 1\le j\le d.
\end{equation}
Multiplying the congruence~\eqref{eq:cong11} by $\ell$, using~\eqref{eq:123} and the fact that $(b_1,\ldots ,b_d)\in \cB$, we see that there exists an integer $w$ such that $W$ is bounded by the number of solutions to 
\begin{equation}
\begin{split}
\label{eq:bbh-b5d=}
\sum_{j=1}^{d}b_j\(n^{j}-m^{j}\)=w & \mand  f(m)\not \equiv f(n) \mod{m},
\end{split} 
\end{equation}
with variables satisfying $1\le m,n\le H$.
Note the condition $f(m)\not \equiv f(n) \mod{m}$ follows from~\eqref{eq:cong11} and the fact that $\lambda \not \equiv 0 \mod{m}$. If $W\neq 0$ then we may assume $w\neq 0$, since if $W\neq 0$ and $w=0$ then from~\eqref{eq:bbh-b5d=}, there exists some pair $(m,n)$ satisfying
$$
0\equiv \sum_{j=1}^{d}b_j(n^{j}-m^{j}) \equiv f(n)-f(m) \not \equiv 0 \mod{m},
$$
giving a contradiction. 

Define 
$$\cM=\{ (n,m):~\text{\eqref{eq:bbh-b5d=} holds}\}.$$
Fix some $(n_0,m_0)\in \cM$ and consider
$$
\cM^{*}=\left\{ \(n^i-m^i-(n_0^i-m_0^i\)_{i=1}^d \in \Z^{d} :~(n,m)\in \cM\right  \}.
$$
Let $\cV$ denote the vector space over $\Q$ generated by $\cM^{*}$.  Note that
$$
\dim \cV<d,
$$
since every point of $\cV$ is orthogonal to $(b_1,\ldots ,b_d)$. Define $d_0$ by 
$$\dim \cV=d_0.$$  For each $1\le j \le d_0$ choose nonzero points 
$$\tilde{v}_j= \(n_j^i-m_j^i-(n_0^i-m_0^i\)_{i=1}^d\in \cM^{*},$$
such that $\tilde{v}_1,\ldots ,\tilde{v}_{d_0}$ are linearly independent over $\Q$ and generate $\cV$ as a vector space. 
Let $M$ denote the matrix whose columns are $\tilde{v}_1,\ldots ,\tilde{v}_{d_0}$ and consider the orthogonal complement
$$
\cV^{*}=\{ x=(x_1,\ldots ,x_d)\in \Q^{d} :~xM=0\}.
$$

By Lemma~\ref{lem:BomVaa}, we see that $\cV^{*}$ has a basis $\tilde w_1,\ldots , \tilde w_{d-d_0}$ which satisfies 
\begin{equation}
\label{eq:wbounds-1}
\tilde w_j=(w_{j,1},\ldots ,w_{j,d}), \quad w_{j,i}=H^{O(1)}, 
\end{equation}
and 
\begin{equation}
\label{eq:wijeqn}
\sum_{i=1}^{d}w_{j,i}(n^{i}-m^{i})=\sum_{i=1}^{d}w_{j,i}(n_0^{i}-m_0^{i}),
\end{equation}
for each $(n,m)\in \cM$ and $1\le j \le d-d_0$. We next show there exists some $1\le j \le d-d_0$ such that 
$$
\sum_{i=1}^{d}w_{j,i}(n_0^{i}-m_0^{i})\neq 0.
$$
Since $(b_1,\ldots ,b_d)\in \cV^{*}$, there exists $\lambda_1,\ldots ,\lambda_{d-d_0}\in \Q$ such that 
$$
(b_1,\ldots ,b_d)=\sum_{j=1}^{d-d_0}\lambda_j \tilde w_{j}.
$$
Hence 
$$
\sum_{i=1}^{d}b_i(n_0^{i}-m_0^{i})=\sum_{j=0}^{d-d_0}\lambda_j\sum_{i=1}^{d}w_{j,i}(n_0^{i}-m_0^{i}).
$$
Since $(m_0,n_0)\in \cM$, we have 
$$0\neq w=\sum_{i=1}^{d}b_i(n_0^{i}-m_0^{i}),$$
which combined with the above implies that 
$$w^{*}=\sum_{i=1}^{d}w_{j_0,i}(n_0^{i}-m_0^{i})\neq 0,$$
for some $1\le j_0\le d-d_0$. From~\eqref{eq:wijeqn}, we see that each pair $(m,n)\in \cM$ satisfies 
\begin{equation}
\label{eq:wj0}
\sum_{i=1}^{d}w_{j_0,i}(n^{i}-m^{i})=w^{*}.
\end{equation}
By~\eqref{eq:wbounds-1}, we have $w^{*}=H^{O(1)}$ and the result follows from Lemma~\ref{lem:poly-eqn} provided $w_{j_0,i}\neq 0$ for some $i\ge 2$.  If $w_{j_0,i}=0$ for all $i\ge 2$ then~\eqref{eq:wj0} implies that each pair $(m,n)\in \cM$ satisfies 
\begin{equation}
\label{eq:w**}
n=m+w^{**},
\end{equation}
for some $w^{**}\in \N$. Substituting~\eqref{eq:w**} into~\eqref{eq:bbh-b5d=}, and recalling that by~\eqref{eq:123} $b_d\not=0$, we obtain that $W=O(1)$ which completes the proof.
\end{proof}
\subsection{Small fractional parts of linear forms}

\begin{lemma}
\label{lem:fractional-linear}
Let $d$ be an integer and $m_{i,j}, 1\le i,j\le d$, be a sequence of integers such that the matrix $M$ with $(i,j)$-th entry $m_{i,j}$ satisfies 
$$\det{M}\neq 0.$$
For any $\varepsilon_1,\ldots,\varepsilon_d> 0$
\begin{align*}
\mu\(\{(t_1,\ldots,t_d)\in [0,1]^{d}:~\{m_{1,j}t_1+\ldots+m_{d,j}t_d\}\le \varepsilon_j, \  1\le j \le d\}\) \ &
\\  \ll \varepsilon_1  \ldots  \varepsilon_d &,
\end{align*} 
where $\mu$ denotes the Lebesgue measure.
\end{lemma} 

\begin{proof}
For each $1\le j \le d$, let $f_j$ be a $1$-periodic function satisfying 
$$f_j(x)\gg 1, \quad |x|\le \varepsilon_j,$$
with Fourier expansion 
$$f_j(x)=\sum_{\ell\in \Z}a_{j,\ell}e(\ell x),$$
and 
$a_{j,\ell}$ satisfying 
\begin{equation}
\label{eq:123123123}
a_{j,0}\ll \varepsilon_j, \quad \sum_{\ell}|a_{j,\ell}|< \infty.
\end{equation}
Define 
$$
\cS = \{(t_1,\ldots,t_d)\in [0,1]^{d}: ~\{m_{1,j}t_1+\ldots+m_{d,j}t_d\}\le \varepsilon_j, \ \ 1\le j \le d\},$$
so that 
$$
\mu(\cS)\ll \int_{[0,1]^{d}}\prod_{j=1}^{d}f_j(m_{1,j}t_1+\ldots+m_{d,j}t_d)dt_1\ldots dt_{d}.
$$
Expand each $f_{j}$ into a Fourier series to get 
$$
\mu(\cS)\ll \sum_{\ell_1,\ldots, \ell_{d}\in \Z}a_{1,\ell_1}\ldots a_{d,\ell_d}\int_{[0,1]^{d}}e\left(\sum_{r=1}^{d}t_r(\sum_{j=1}^{d}m_{r,j}\ell_j )\right)dt_{1}\ldots dt_{d}.
$$
Since each $m_{i,j}\in \Z$, by orthogonality of the exponential function 
$$
\mu(\cS)\ll \sum_{\substack{\ell_1,\ldots, \ell_{d}\in \Z \\ \sum_{j=1}^{d}m_{r,j}\ell_j=0 \\ 1\le r \le d}}|a_{1,\ell_1}\ldots a_{d,\ell_d}|.
$$
Using the assumption $\det M\neq 0$, the only term in the above sum corresponds to $(\ell_1,\ldots,\ell_d)=0$ and the result follows from~\eqref{eq:123123123}.
\end{proof}

\subsection{Bilinear character sums and energy} 
To prove  Theorem~\ref{thm:Primes}, we require the bound of~\cite[Corollary~1.3]{ShkShp} on the bilinear multiplicative character sums
$$
W_\chi(\cI, \cS;\balpha, \bbeta) := \sum_{s\in \cS}  \sum_{x \in \cI} \alpha_s \beta_x    \chi(s+x) , 
$$
where $\cI = [1,H]\subset \Z_p$, $\cS \subseteq \Z_p$ is  an arbitrary  set of cardinality $\# \cS = S$, 
and  $\balpha = \{\alpha_s\}_{s\in \cS},$ $\bbeta = \{\beta_x\}_{x\in \cI}$  
are sequence of complex weights with
\begin{equation}
\label{eq:unit disc}
|\alpha_s| \le 1, \quad s\in \cS, \mand |\beta_x| \le 1, \quad x\in \cI. 
\end{equation}

\begin{lemma}
\label{cor:Gen Set Energy} Given an interval  $\cI = [1,H]$ 
of length $H<p$ 
and a set $\cS \subseteq\Z_p^*$  of size $\# \cS=S$
with  
\begin{equation}
\label{eq:cond SX}
S^2 H \le  p^{2} \mand H < p^{1/2}
\end{equation}
and complex weights $\balpha = \{\alpha_s\}_{s\in \cS}$ and $\bbeta = \{\beta_x\}_{x\in \cI}$  
satisfying~\eqref{eq:unit disc},
for any fixed integer $r \ge 1$ such that $H \ge p^{1/r}$,  we have
\begin{align*}
& W_\chi(\cI, \cS;\balpha, \bbeta)    \\
&\qquad  \ll  SH\( \frac{E(\cS) p^{(r+1)/r}}{S^4H^2}+\frac{ p^{(r+2)/r}}{  SH^{5/2}} +\frac{p^{(r+2)/r}}{S^2 H^2} \)^{1/4r} p^{o(1)} + S^{1/2} H\,.
\end{align*}
\end{lemma}

 \section{Proofs of Main Results}
 
 \subsection{Proof of Theorem~\ref{thm:Energy1}}

For an integer $s \ge 1$ we denote by  $\sfT_s$ the number of solutions to the congruence 
\begin{equation}
\begin{split}
\label{eq:sym cong}
f(x_1)+\ldots+f(x_s)&\equiv f(x_{s+1})+\ldots+f(x_{2s}) \mod  m,
\end{split}
\end{equation}
with variables satisfying  $x_1\ldots,x_{2s} \in \cI,$ so that
$$
 T_{f,m}(\cI)  = \sfT_2. 
$$

Now we define $K \ge 1$ by the equation 
\begin{equation}
\label{eq: K def}
\sfT_2 =  H^3/K .
\end{equation}

Let 
$$
R(\lambda) =  \# \{\(x_1,x_2\) \in \cI^2 :~f(x_1)+f(x_2) = \lambda\} .
$$
Using that 
$$
H^2 =\sum_{\lambda \in \Z_m}R(\lambda)   = \sum_{\lambda \in f(\cI) + f(\cI)}  R(\lambda)  , 
$$
by the Cauchy--Schwarz inequality, we derive from~\eqref{eq: K def} that
\begin{align*}
H^4 & \le   \# \(f(\cI) + f(\cI)\) \sum_{\lambda \in f(\cI) + f(\cI)} R(\lambda)^2 \\
& =  \# \(f(\cI) + f(\cI)\) \sfT_2 =  \# \(f(\cI) + f(\cI)\) H^3/K.
\end{align*}
Therefore,  we have 
\begin{equation}
\label{eq: Bound HK}
 \# \(f(\cI) + f(\cI)\)  \ge HK.
\end{equation}
We now obtain a lower bound on $K$.

It has been shown in~\cite[Equation~(2.26)]{ShkShp}  that expressing   $ \sfT_s$ 
via exponential sums and using the H\"older inequality, for any integer $s \ge 2$
we have 
$$
\sfT_2^{s-1}  \le \sfT_s\(\# f(\cI)\)^{s-2}  \ll H^{s-2} \sfT_s. 
$$
Hence  from~\eqref{eq: K def} we derive
\begin{equation}
\label{eq:T Low}
\sfT_s \gg H^{2s-1}/ K^{s-1}.
\end{equation}

From now on the proof essentially follows that of~\cite[Theorem~1]{CGOS}.
Namely, writing
\begin{equation}
\label{eq:lambda}
\lambda_j = x_1^{j}+\ldots + x_s^{j}-x_{s+1}^j-\ldots-x_{2s}^j,
\qquad j = 1, \ldots, d,
\end{equation}
we see that $\lambda_j \in [-sH^j, sH^j]$.
We also see that for any solution to~\eqref{eq:sym cong} we have the linear 
congruence 
$$
a_d \lambda_d +\ldots +a_1\lambda_1 \equiv 0 \mod  m
$$
For each of $O(H^{d(d-1)/2})$ choices of
$$(\lambda_1, \ldots, \lambda_{d-1}) \in
 [-sH, sH] \times \ldots  [-sH^{d-1}, sH^{d-1}],
$$
we see that for any solution to~\eqref{eq:sym cong} we have
$$
a_d \lambda_d \equiv \xi \mod  m
$$
for some $\xi $ depending only on 
$\lambda_1, \ldots, \lambda_{d-1}$.
Since $\gcd(a_d,m) =1$, we conclude that $\lambda_d$ can take $O(H^d/m + 1)$
possible values.
We now see that for some integers $\mu_1, \ldots, \mu_d$
\begin{equation}
\label{eq:T and J}
\sfT_s \ll \(H^d/m + 1\)H^{d(d-1)/2}
J_{d,s}(\mu_1, \ldots,\mu_d;H),
\end{equation}
where $I_{d,s}(\lambda_1, \ldots, \lambda_d; H)$ is
the number of solutions to the system of equations~\eqref{eq:lambda}
in variables  $1 \le x_1\ldots,x_{2s}\le H$.
As in~\cite[Equation~(13)]{CGOS}, a  standard argument shows that 
\begin{equation}
\label{eq:J(0)}
I_{d,s}(\lambda_1,   \ldots, \lambda_d; H) \le I_{d,s}(0,   \ldots, 0; H)
= J_{d,s}(\{1, \ldots, H\}).
\end{equation}
 
Thus, combining~\eqref{eq:T and J} and~\eqref{eq:J(0)}, we derive
$$
\sfT_s \ll  (H^d/m + 1)H^{d(d-1)/2+1}  J_{d,s}(\{1, \ldots, H\}).
$$
Hence, Lemma~\ref{lem:MVT-GenSet}, taken with  
$$
s = d(d+1)/2 \mand \cX = \{1, \ldots, H\}
$$ 
yields
\begin{equation}
\label{eq:T Up}
\sfT_s \le 
 (H^d/m + 1)H^{2s - d +1+o(1)} =  H^{2s +o(1)}/m + H^{2s - d +o(1)}.
\end{equation}
Combining~\eqref{eq:T Low} and~\eqref{eq:T Up} 
we obtain 
$$
K^{-s+1} H^{2s-1} \le H^{2s +o(1)}/m + H^{2s - d +o(1)}
$$
or 
$$
K^{-s+1}  \le H^{1 +o(1)}/m + H^{ - d+1 +o(1)}.
$$ 
With the above choice of $s$, we obtain 
$$ 
K \ge \min\left\{\(m/H\)^{\alpha_d}, H^{\beta_d}\right\} H^{o(1)}, 
$$
where $\alpha_d$ and $\beta_d$ are given by~\eqref{eq: alpha beta},
which together with~\eqref{eq: Bound HK}  concludes the proof. 

 \subsection{Proof of Theorem~\ref{thm:Energy2}}
Define $\rho(w)$ to be the number of solutions to the congruence 
\[
w\equiv u - v \mod {m}, \qquad u,v \in f(\cI).
\]

Let $r(w)$ denote the number of solutions
\begin{equation}
\label{eq:Cong wfxfy}
 w\equiv f(x) - f(y) \mod {m}, \qquad x,y \in \cI,
\end{equation} 
and let $\cX_w$ denote the set of $x\in \cI$ such that there exists $y\in \cI$ satisfying the above congruence. 
Trivially, for $w \not \equiv 0\mod {m}$ we have 
\begin{equation}
\label{eq:fxfy}
 f(x)\not \equiv f(y) \mod {m}
\end{equation} 
for each solution $(x,y)$ to~\eqref{eq:Cong wfxfy}. 
It is clear that $\rho(w)\leq r(w)$ for all $w\in\Z_m$ and
\begin{equation}
\label{eq:ED}
E_+(\cZ,f(\cI))= HZ +\sum_{w=1}^{m-1}R (w)\rho(w),
\end{equation} 
where $R(w)$ counts the number of solutions to 
$$z_1-z_2=w, \quad z_1,z_2\in \cZ.$$

Note that  by Lemma~\ref{lem:short}, if $H\le m^{c_d}$ for some constant $c_d$ then $r(w)=H^{o(1)}$, so that~\eqref{eq:ED} implies 
\begin{align*}
E_+(\cZ,f(\cI))&= HZ +H^{o(1)}\sum_{w=1}^{m-1}R (w) \\ 
&= HZ +H^{o(1)}Z^2.
\end{align*}
It is sufficient to assume $H\ge m^{c_d}$ and hence establish that 
$$
E_+(\cZ,f(\cI)) = \(\frac{H^{2}Z^2}{m^{2/d(d+1)}}+Z(H+Z)\)m^{o(1)}.
$$ 
Assume without loss of generality, that 
$$
a_d = 1 \mand a_0 = 0.
$$ 

Fix $1\le w \le m-1$ and consider  $r(w)$.

Let $s=d(d+1)/2$. We define the intervals
$$
\cI_{j}=[-s H^j, s H^j],
\qquad  j =1, \ldots, d,
$$
and consider the set
$$
\cS_w\subseteq  \cI_{1}\times \ldots\times \cI_{d}\subseteq \Z^{d},
$$
of all the $d$-tuples
\begin{equation}
\label{eq:triples} \vec{x}=
(x_1+\ldots +x_{s}, \ldots, x_1^d+\ldots +x_{s}^d),
\end{equation}
such that $x_i\in \cX_w$.

For each $\vec{x}$ we let $I(\vec{x})$ count the number of representations of $\vec{x}$
in the form~\eqref{eq:triples} with variables $x_1,\ldots,x_s\in \cX$. Thus
$$\sum_{\vec{x} \in \cS_w} I(\vec{x}) = X_w^{s},$$
where 
$$
X_w = \# \cX_w
$$
and 
$$\sum_{\vec{x} \in \cS_w} I(\vec{x})^2 = J_{d,s}\(\cX_w\),$$
with $J_{d,s}\(\cX_w\)$ is as in Section~\ref{sec:MVT}.   By Lemma~\ref{lem:MVT-GenSet} we have, 
$$\sum_{\vec{x} \in \cS_w} I(\vec{x})^2\le X_w^{s} H^{o(1)},$$
and from the Cauchy--Schwarz inequality
$$
X_w^{s}=\sum_{\vec{x} \in \cS_w} I(\vec{x}) \le \(\# \cS_w
\sum_{\vec{x} \in \cS_w} I(\vec{x})^2\)^{1/2},
$$
we see that
\begin{equation}\label{eq:S large}
\# \cS_w\ge X_w^{s}H^{o(1)}.
\end{equation}

Define the lattice 
\begin{align*}
\cL=\biggl\{(v_1, \ldots, v_d,w_1,\ldots,w_d)& \in \Z^{2d} :\\
 ~\sum_{i=1}^d & a_i(v_i-w_i)  \equiv 0 \mod {m} \biggr\},
\end{align*}
where $a_1, \ldots, a_d$ are non-constant coefficients of $f$ in~\eqref{eq:Poly f}
and the convex body
$$
\cD= \{(g_1,\ldots,g_d, h_1, \ldots, h_d)  \in \R^{2d}  :~
 ~|g_i|, |h_i|\le sH^{i},  \  i =1, \ldots, d\}.
$$

It follows from~\eqref{eq:S large} that
\begin{equation}
\label{eq:Xl1}
X_w^{s}\ll \#\(\cL \cap \cD\)H^{o(1)}.
\end{equation}

Let $\lambda_1,\ldots,\lambda_{2d}$ denote the successive minima of $\cL$ with respect to $\cD$. 

We consider two cases depending on if $\lambda_{2d}<1$ or not.

 First consider the case
$$
\lambda_{2d}<1.
$$ 
By Lemmas~\ref{lem:mst} and~\ref{lem:lattice}
$$
\#\(\cL \cap \cD\)\ll \frac{H^{d(d+1)}}{m},
$$
and hence by~\eqref{eq:Xl1} we have
$$
X_w\le \frac{H^{2+o(1)}}{m^{1/s}} , 
$$ 
which by~\eqref{eq:Xl1} implies 
$$
\rho(w) \le \frac{H^{2+o(1)}}{m^{1/s}}. 
$$

Consider next when 
\begin{equation}
\label{eq:case2}
\lambda_{2d}\ge 1.
\end{equation}
Recalling~\eqref{eq:Gammastar} and~\eqref{eq:Dstar}, we see that the dual lattice $\cL^{*}$  and the dual body $\cD^{*}$ are given by
\begin{align*}
\cL^{*}= \frac{1}{m}\Bigl\{(u_1, \ldots, u_d, z_1, &\ldots,z_d) \in \Z^{2d}: \\
& a_i u_d\equiv -z_i \equiv u_i \mod {m}, \ i =1, \ldots, d\Bigr\},
\end{align*}  
and
$$
\cD^{*}=\left \{ \(g_1, \ldots, g_d, h_1,\ldots,h_d\)\in \R^{2d}:~ \sum_{i=1}^{d}sH^{i}\(|h_i|+ |g_i|\)\le 1\right  \}. 
$$

If $\lambda_1^{*}$ denotes the first successive minimum of $\cL^{*}$ with respect to $\cD^{*}$, then by~\eqref{eq:case2} and Lemma~\ref{lem:transfer} there exists some constant $c$ depending only on $d$ such that
$$
\lambda_1^{*}\le c 
$$
which in turn implies 
$$
\cL^{*}\cap c\cD^{*}\neq \emptyset.
$$
Hence there exist $u_1, \ldots, u_d\in \Z$ such that 
\begin{equation}
\label{eq:congcond}
a_iu_d\equiv u_i \mod {m}, 
\end{equation}
and
\begin{equation}
\label{eq:sizecond}
|u_i|\ll \frac{m}{H^{i}}, \quad 1\le i \le d.
\end{equation}
Recalling~\eqref{eq:fxfy}, we now see that 
\begin{align*}r(w)\le \# \bigl\{ x,y\in \cI:~\widetilde f(x)  - \widetilde f(y)  &\equiv \mu \mod {m}, \\
  & f(x)\not\equiv f(y) \mod {m} \bigr\} ,
  \end{align*}
where $\mu \equiv u_d w \mod  m$, $0 \le \mu < m$,  and 
$$\widetilde f(x)=u_1x+ u_2 x^2+\ldots+u_dx^{d}.$$

Recalling~\eqref{eq:congcond}, we obtain
$$
\widetilde f(x)-\widetilde f(y)=  \mu + tm,
$$
for some $t\in \Z$.

If $x,y \in \cI$, we see from~\eqref{eq:sizecond}  that 
$$
f(x)-f(y) \ll m
$$
and thus $t$ takes only $O(1)$ possible values.
 Hence there exists some integer $\mu_0$ such that 
$$r(w)\ll \# \left\{ x,y\in \cI :~  \widetilde f(x)-\widetilde f(y)= \mu_0,  \  f(x)\not\equiv f(y) \mod {m}  \right \}.$$ 
Note if $\mu_0= 0$ 
 then the conditions  $f(x)\not\equiv f(y) \mod {m}$ and $1\le u_d\le m-1$ imply that $r(w)=0$. Hence we may suppose $\mu_0\neq 0$.

It follows from Lemma~\ref{lem:poly-eqn} and the inequality $H \le m$ that  
\begin{equation}\label{eq:case2final}
r(w)\le w^{o(1)}\le   (mH)^{o(1)} = m^{o(1)}.
\end{equation}
Combining~\eqref{eq:case2final} with~\eqref{eq:ED} and using
$$
\sum_{w=1}^{m-1}R(w) \le Z^2,
$$
 gives 
\begin{align*}
E_+(\cZ,f(\cI))&=HZ+\left(\frac{H^{2}}{m^{1/s}}+1\right)m^{o(1)}\sum_{w=1}^{m-1}R(w) \\ 
&=\(\frac{H^{2}Z^2}{m^{2/d(d+1)}}+Z(H+Z)\)m^{o(1)},
\end{align*}
which concludes the proof.


\subsection{Proof of Theorem~\ref{thm:main3}}
As in the proof of Theorem~\ref{thm:Energy2}, we may suppose $H\ge m^{c_d}$ for some constant $c_d$ and  that 
$$
a_d = 1, \quad a_0 = 0.
$$
Consider the lattice 
$$
\cL=\left\{(v_1, \ldots, v_d) \in \Z^{d} : ~\sum_{i=1}^d  a_i v_i  \equiv 0 \mod {m} \right\},
$$
and the convex body
$$
\cD= \{(g_1,\ldots,g_d)  \in \R^{d}:  ~ |g_i|\leq H^{i},  \  i =1, \ldots, d \}.
$$ 
Let $\lambda_1,\ldots,\lambda_d$ denote the successive minima of $\cL$ with respect to $\cD$. We consider two cases depending on whether
\begin{equation}
\label{eq:case32}
\lambda_d>1
\end{equation} 
or 
\begin{equation}
\label{eq:case31}
\lambda_d\le 1.
\end{equation}

If~\eqref{eq:case32} holds, then there exists a constant $c$ depending only on $d$ such that 
for the dual lattice and body we have
$$\cL^{*}\cap c\cD^{*}\neq 0.$$
Arguing as in the proof of Theorem~\ref{thm:Energy2}, we obtain a polynomial 
$$g(x)=u_1x+\ldots+u_dx^{d},$$
with $u_d\neq 0$ and coefficients bounded by $m$ such that
$$T_{f,m}(\cI)\ll \# \{ 1\le x,y,z,w \le H :~ g(x)-g(y)=g(w)-g(z) \}.$$
Hence from Corollary~\ref{cor:poly-eqn}
$$
T_{f,m}(\cI)\le H^{2} m^{o(1)}.
$$

Consider next~\eqref{eq:case31}. By Lemma~\ref{lem:mahler}, there exists a basis
$$\vec{v}_{j}=(v_{1,j},\ldots,v_{d,j})\in \cL\cap \lambda_j\cD,$$
such that each $\vec{b}  \in \cL\cap \cD$ can be expressed in the form 
\begin{equation}
\label{eq:m-basis}
\vec{b}=b_1\vec{v}_1+\ldots+ b_d \vec{v}_d, \qquad  b_j\ll \frac{1}{\lambda_j}, \ j =1, \ldots, d.
\end{equation}

For an integer $x$ we use $\widetilde x$ to denote the following vector
$$\widetilde x=(x,x^2,\ldots,x^{d}).$$

Let $\varphi$ be a smooth function  with support 
$$
\supp \varphi \subseteq [-2,2]
$$
and 
satisfying 
$$\varphi(x)\gg 1, \quad |x|\le 1.$$
Let $K$ be a large integer to be determined later. Using orthogonality of exponential functions,  it follows that for a constant $c$, depending on the implied constants in~\eqref{eq:m-basis}, and hence only on $d$,
\begin{align*}
T_{f,m}(\cI)\ll&\frac{1}{K^{d}}\sum_{h_1, \ldots, h_d\in \Z}\sum_{x_1, \ldots, x_4\le H}\prod_{j=1}^{d}\varphi(c\lambda_j h_j)\\
&\qquad \int_{\vec{y}\in [0,K]^{d}}\e\(\langle \widetilde x_1+\widetilde x_2-\widetilde x_3-\widetilde x_4,\vec{y}\rangle +\sum_{j=1}^{d}h_j\langle \vec{v}_j,\vec{y}\rangle\)d\vec{y},
\end{align*}
where, as in the above $\langle \ast, \ast \rangle$ denotes the Euclidean inner product and
$$\vec{y}=(y_1,\ldots,y_d), \quad \widetilde x_j=(x_j,\ldots,x_j^{d})$$
(as in the above convention).

Note that we have used periodicity of the exponential function to integrate over the large cube $[0,K]^{d}$. Our next step is to reduce the above to a counting problem involving a fundamental domain for $\cL^{*}$ and having such a large range of integration simplifies some technical issues.

Consider first summation over $h_1, \ldots, h_d$.  By Poisson summation 
$$
\sum_{h_j\in \Z}\varphi(c\lambda_j h_j)\e(h_j\langle \vec{v}_j,\vec{y}\rangle)\ll \frac{1}{\lambda_j}\sum_{h\in \Z}\left|\widehat \varphi\(\(\langle \vec{v}_j,\vec{y}\rangle -h\)/\lambda_j\)\right|.
$$
Repeated integration by parts shows that for any small $\varepsilon>0$ and large constant $C$
$$
\widehat \varphi\(x/\lambda_j\) \ll \frac{1}{x^2 m^{C}} \quad \text{unless} \quad |x|\le \lambda_j m^{\varepsilon}.
$$
Since $\lambda_1,\ldots,\lambda_d\le 1$, this implies 
\begin{align*}
\sum_{h_j\in \Z}\varphi\(c\lambda_j h_j\)\e\(h_j\langle \vec{v}_j,\vec{y}\rangle\)\ll \begin{cases}\lambda_j^{-1} & \text{if} \quad \{ \langle \vec{v}_j,\vec{y}\rangle\}\le \lambda_j m^{\varepsilon}, \\ m^{-C} &\text{otherwise} \end{cases}
\end{align*}
where $\{\xi\}$ denotes the fractional part of $\xi \in \R$.

Define the set 
\begin{equation}
\label{eq:V*}
\cV^{*}=\left\{ \vec{y}\in [0,K]^{d} :~ \{ \langle \vec{v}_j,\vec{y}\rangle\}\le \lambda_jm^{\varepsilon}, \   j=1, \ldots, d\right\},
\end{equation}
and let $\mu$ denote the Lebesgue measure in $\R^d$. The above implies 
$$
T_{f,m}(\cI)\ll \frac{1}{\lambda_1\ldots \lambda_d}\frac{1}{K^{d}}\int_{\vec{y}\in \cV^{*}}\left|\sum_{1\le x \le H}\e(\langle \widetilde x, \vec{y} \rangle ) \right|^{4} d\mu\(\vec{y}\)
$$
and hence by H\"{o}lder's inequality, we have
\begin{align*}
T_{f,m}(\cI)& \ll \frac{1}{\lambda_1\ldots \lambda_d}  \frac{\mu(\cV^{*})^{1-4/d(d+1)}}{K^{d}}\\
&\qquad \qquad  \(\int_{[0,K]^{d}}\left|\sum_{1\le x \le H}
\e(\langle \widetilde x, \vec{y} \rangle  d\mu\(\vec{y}\))  \right|^{d(d+1)}\)^{4/d(d+1)}.
\end{align*}  
By Lemma~\ref{lem:vmvt} and the fact that 
\begin{align*}
\int_{[0,K]^{d}}\left|\sum_{1\le x \le H}
\e(\langle \widetilde x, \vec{y} \rangle  d\mu\(\vec{y}\))  \right|^{d(d+1)}&\\
=K^{d}\int_{[0,1]^{d}}&\left|\sum_{1\le x \le H}
\e(\langle \widetilde x, \vec{y} \rangle  d\mu\(\vec{y}\))  \right|^{d(d+1)},
\end{align*}
 we derive
\begin{equation}
\label{eq:Tfm-1}
T_{f,m}(\cI)\le\frac{1}{\lambda_1\ldots \lambda_d}\frac{\mu(\cV^{*})^{1-4/d(d+1)}}{K^{d(1-4/d(d+1))}}H^{2+o(1)}.
\end{equation}
We next estimate the measure $\mu(\cV^{*}).$ By Lemma~\ref{lem:mahler}, we may choose a basis $\vec{v}_1^{*},\ldots,\vec{v}_{d}^{*}$ for $\cL^{*}$ satisfying 
\begin{equation}
\label{eq:vecjstar}
\vec{v}_j^{*}\in \frac{d\lambda_j^{*}}{2}\cD^{*}.
\end{equation}
Consider the fundamental domain
\begin{equation}
\label{eq:funddomain1}
\R^{d}/\cL^{*}=\{ t_1\vec{v}_1^{*}+\ldots+t_d \vec{v}_{d}^{*} :~ 0\le t_1, \ldots, t_d<1\}.
\end{equation}
We have 
\begin{equation}
\label{eq:funddomain}
\mu(\R^{d}/\cL^{*})=\det \cL^{*}=\frac{1}{\det \cL }=\frac{1}{m}.
\end{equation}
Let $\cL^{*}_0$ denote the set 
\begin{equation}
\label{eq:l0l0l0}
\cL_0^*=\{ \vec{b} \in \cL^{*} :~ \vec{b}+\R^{d}/\cL^{*}\cap [0,K]^{d}\neq \emptyset\}.
\end{equation}

By~\eqref{eq:vecjstar} and the fact that 
$$\cD^{*}\subseteq \prod_{j=1}^{d}\left[-\frac{c_0}{H^{j}},\frac{c_0}{H^{j}}\right]
$$
for some absolute constant $c_0$.
It follows from~\eqref{eq:funddomain1} that
$$
\R^{d}/\cL^{*}\subseteq \prod_{j=1}^{d}\left[-\frac{c_0\lambda_d^{*}}{H^{j}},\frac{c_0\lambda_d^{*}}{H^{j}}\right],
$$  
which by Lemma~\ref{lem:transfer} implies 
$$
\R^{d}/\cL^{*}\subseteq \prod_{j=1}^{d}\left[-\frac{c_1}{\lambda_1H^{j}},\frac{c_1}{\lambda_1H^{j}}\right],
$$
for some constant $c_1$ depending only on $d$.
Since $$\cL\cap \lambda_1\cD\neq \{0\},$$ it follows that 
$$\lambda_1\gg\frac{1}{H^{d}},$$
which implies 
$$
\R^{d}/\cL^{*}\subseteq \prod_{j=1}^{d}\left[-c_2H^{d-j},c_2H^{d-j}\right].
$$
Recall~\eqref{eq:l0l0l0} and note that if $\vec{b}\in \cL_0^{*}$ then 
$$\vec{b}\in [0,K]^{d}-\R^{d}/\cL^{*}.$$
Hence
$$
\cL_0^{*}\subseteq [-c_0H^{d},K+c_0H^{d}]^{d},$$
for some  constant $c_2$ depending only on $d$.
Choosing 
$$K=e^{mH},$$
we see that 
$$
\# \cL_0^{*}\le \# \([-C\log{K},K+C\log{K}]^{d}\cap \cL^{*}\)=m(1+o(1))K^{d},
$$
as $K\rightarrow \infty$.

Note that each $x\in \R^{d}$ has a unique expression in the form
$$x=\ell+\widetilde t, \qquad \ell \in \cL^{*}, \quad \widetilde t\in \R^{d}/\cL^{*},$$
and hence
$$
\mu(\cV^{*})=\sum_{\ell \in \cL_0}\mu(\cV^{*}\cap (\ell+\R^{d}/\cL^{*})).
$$
For any $\ell \in \cL^{*}$ 
$$\mu(\cV^{*}\cap (\ell+\R^{d}/\cL^{*}))=\mu(\cV^{*}\cap (\R^{d}/\cL^{*})),$$
so that the above implies 
$$
\mu(\cV^{*})\le (1+o(1))mK^{d}\mu(\cV^{*}\cap (\R^{d}/\cL^{*})).
$$
Combining with~\eqref{eq:Tfm-1}
\begin{equation}
\label{eq:V*-1}
T_{f,m}(\cI)\le \frac{H^{2+o(1)}}{\lambda_1\ldots \lambda_d}m^{1-4/d(d+1)}\mu(\cV^{*}\cap (\R^{d}/\cL^{*}))^{1-4/d(d+1)}.
\end{equation}
Recall~\eqref{eq:funddomain1} and~\eqref{eq:V*}. If 
$$z\in \cV^{*}\cap (\R^{d}/\cL^{*})$$
then 
\begin{equation}
\label{eq:vjz-1}
\{\langle \vec{v}_j,z\rangle\}\le \lambda_jm^{\varepsilon}, \   1\le j \le d
\end{equation}
and there exists $0\le t_1,\ldots,t_d\le 1$ such that 
$$z=t_1\vec{v}_1^{*}+\ldots+t_d\vec{v}_d^{*}.$$

Let $\cV^{**}$ denote the set of all $0\le t_1,\ldots,t_d\le 1$ such that 
$$
\left\{\sum_{k=1}^{d}t_k\langle \vec{v}_j,\vec{v}^{*}_k\rangle \right\}\le \lambda_j m^{\varepsilon}, \quad 1\le j \le d.
$$ 
Combining~\eqref{eq:vjz-1} with a linear change of variables, and then recalling~\eqref{eq:funddomain},  we see that
\begin{equation}
\label{eq:V**}
\mu\(\cV^{*}\cap (\R^{d}/\cL^{*})\)=\mu(\cV^{**}) \mu(\R^{d}/\cL^{*}) =\frac{\mu(\cV^{**})}{m}.
\end{equation}
Let $M$ denote the matrix with $(j,k)$-th entry 
$$m_{j,k}=\langle \vec{v}_j,\vec{v}^{*}_k\rangle.$$
Note that 
$$m_{j,k}\in \Z.$$
Let $L$ denote the matrix whose rows correspond to $\vec{v}_j$ and $L^{*}$ denote the matrix whose rows correspond to $\vec{v}_{k}^{*}$, so that 
$$M=LL^{*}.$$
Since $\vec{v}_{1},\ldots,\vec{v}_{d}$ form a basis for $\cL$ and $\vec{v}_{1}^{*},\ldots,\vec{v}_{d}^{*}$ form a basis for $\cL^{*}$, the above implies 
$$
\det{M}\neq 0.
$$
Hence by Lemma~\ref{lem:fractional-linear}
$$
\mu(\cV^{**})\ll \lambda_1\ldots \lambda_d m^{d\varepsilon},
$$
which combined with~\eqref{eq:V**} implies 
$$
\mu\(\cV^{*}\cap (\R^{d}/\cL^{*})\)\ll\frac{\lambda_1\ldots \lambda_d}{m^{1-\varepsilon d}}.
$$
The result follows from~\eqref{eq:V*-1} and Lemma~\ref{lem:mst}, after taking $\varepsilon$ sufficiently small.

\subsection{Proof of Theorem~\ref{thm:Primes}}
Since both sums are bounded similarly, we only deal with the first sum. As observed in~\cite{ShkShp}, we may write 
$$
 \sum_{\substack{q\le Q\\q~\text{prime}}}  \left|  \sum_{\substack{r \le R\\ r~\text{prime}}}
 \chi(f(q) +r)\right| = d W_\chi(\cI, \cS;\balpha, \bbeta),
 $$
where $\cS = \{f(q) ~:~q\le Q,\ q~\text{prime}\}$, $\cI = [1, R]$,
$$
\alpha_s = \frac{1}{d} \sum_{\substack{q\le Q,\, f(q) = s\\q~\text{prime}}} e^{i \psi_q}\, \qquad s\in\cS,
$$
with $0 \le \psi_q < 2\pi$ denoting the argument of the inner sum above for each $q$, and $\beta_x$ is the characteristic  function of primes in $\cI$. Writing $\cJ = [1, Q]$ and noting that $E(\cS)\leq E(f(\cJ))\leq d^{-4} T_{f, p}(\cJ)$ and the fact that condition~\eqref{eq:unit disc} holds for our choice of weights, we may use Theorem~\ref{thm:main3} and Lemma~\ref{cor:Gen Set Energy} to bound $W_\chi(\cI, \cS;\balpha, \bbeta)$.

The size restrictions in~\eqref{eq:cond SX} impose the condition 
$$ \xi\le \min\{1/2, 2 - 2 \zeta\}
$$ 
and, for sufficiently large $r$, to ensure that the second and third terms in the bound on $W_\chi(\cI, \cS;\balpha, \bbeta)$ are non-trivial we need the conditions $\zeta + 5\xi/2 >1$ and $\zeta + \xi>1/2$ respectively. The first term is non-trivial for $\xi>1/2 - 2/d(d+1)$ and $\xi+\zeta>1/2$.

The other sum can be estimate similarly after writing it as a bilinear sums. 

This gives the desired bound for sufficiently large $p$. Adjusting the value of $\delta>0$, we can make sure 
it holds for all $p$.

\section*{Acknowledgements}

During the preparation of this work, B.~Kerr was supported by the ARC Grant DE220100859 and A. Mohammadi and
I.~E.~Shparlinski   by the ARC Grant DP200100355.

 \end{document}